\documentclass[11pt]{article}
\usepackage{amsmath,amsthm,amscd,amssymb}
\usepackage{latexsym,epsf,epsfig}

\newtheorem{theorem}{Theorem}[section]
\newtheorem{proposition}[theorem]{Proposition}

\newtheorem{definition}[theorem]{Definition}

\newcommand{\aangle}{\angle\!\!\!)}

\begin{document}

\title{{The Six-Point Circle Theorem}
\thanks{2000 {\it Math Subject Classification.} Primary: 51M25, 51M16. 
Secondary 51M04, 51M15. 
\newline
{\it Key words}: inscribed and circumscribed triangles, pedal triangles, 
geometrical transformations}}

\author{\Large{Adrian Mitrea}}
\date{}

\maketitle

\begin{abstract}
Given $\Delta ABC$ and angles $\alpha,\beta,\gamma\in(0,\pi)$ with 
$\alpha+\beta+\gamma=\pi$, we study the properties of the triangle $DEF$ 
which satisfies: (i) $D\in BC$, $E\in AC$, $F\in AB$, 
(ii) $\aangle D=\alpha$, $\aangle E=\beta$, $\aangle F=\gamma$, 
(iii) $\Delta DEF$ has the minimal area in the class of triangles satisfying 
(i) and (ii). In particular, we show that minimizer $\Delta DEF$, exists, 
is unique and is a pedal triangle, corresponding to a certain pedal point $P$. 
Permuting the roles played by the angles $\alpha,\beta,\gamma$ in (ii), 
yields a total of six such area-minimizing triangles, which are pedal 
relative to six pedal points, say, $P_1,....,P_6$. The main result of the 
paper is the fact that there exists a circle which contains all six points. 
\end{abstract}

\section{Introduction}
Consider the following question in planar geometry:  
{\it given three, nonconcurrent lines, $L_1, L_2, L_3$, along with three 
angles $\alpha,\,\beta,\,\gamma\in(0,\pi)$ satisfying 
$\alpha+\beta+\gamma=\pi$, find three points $X_1,X_2,X_3$ such that}
\begin{eqnarray}\label{madden1}
X_1\in L_1,\,\,X_2\in L_2,\,\,X_3\in L_3,
\end{eqnarray}
\noindent {\it and}
\begin{eqnarray}\label{madden2}
\aangle\,X_2X_1X_3=\alpha,\quad
\aangle\,X_1X_2X_3=\beta,\quad
\aangle\,X_1X_3X_2=\gamma.
\end{eqnarray}
\noindent Somewhat informally, with $A,B,C$ playing the roles of the 
points of mutual intersection of the lines $L_1,L_2,L_3$,  
this asks whether given a fixed triangle $ABC$, 
one can inscribe a triangle $X_1X_2X_3$ of an arbitrarily prescribed shape.
For a triangle, ``shape''  simply signifies the measures of its angles
(in the current scenario $\alpha,\beta,\gamma$). 

Although not necessarily obvious, this question turns out to have an 
affirmative answer. In fact, the collection 
${\mathcal I}_{\alpha\beta\gamma}$ of all triangles $X_1X_2X_3$ satisfying 
(\ref{madden1}) and (\ref{madden2}) has infinite cardinality. Since all 
triangles in ${\mathcal I}_{\alpha\beta\gamma}$ have the {\it same shape}, 
what distinguishes them is their {\it size}. It is not too difficult to 
show that ${\mathcal I}_{\alpha\beta\gamma}$ contains triangles of arbitrarily 
large areas, hence, the most distinguished triangle in 
${\mathcal I}_{\alpha\beta\gamma}$ is the one of the {\it smallest} area, 
called, for the purposes of this introduction, $\Delta_{\alpha\beta\gamma}$. 
Part of our paper is devoted to studying this minimizer, 
which enjoys many interesting properties. Among other things we show 
that this triangle of minimal area is unique, and is also a 
{\it pedal triangle}. Recall that a pedal triangle is a triangle whose 
vertices are the projections of a point, called the pedal point of the pedal 
triangle, onto the sides of another triangle. Let $P_{\alpha\beta\gamma}$ 
be the pedal point of $\Delta_{\alpha\beta\gamma}$. 

Of course, similar considerations apply to the situation when the 
ordered triplet $(\alpha,\beta,\gamma)$ in (\ref{madden2}) is replaced 
by any of its permutations 
\begin{eqnarray}\label{perm-2}
(\alpha,\beta,\gamma),\,\,\,
(\alpha,\gamma,\beta),\,\,\,
(\beta,\alpha,\gamma),\,\,\,
(\beta,\gamma,\alpha),\,\,\, 
(\gamma,\alpha,\beta),\,\,\,
(\gamma,\beta,\alpha). 
\end{eqnarray}
\noindent As before, these permutations lead to considering 
the collections ${\mathcal I}_{\alpha\beta\gamma}$, 
${\mathcal I}_{\alpha\gamma\beta}$, ..., defined as above, which in turn 
have their own minimizers,
\begin{eqnarray}\label{delta1}
\Delta_{\alpha\beta\gamma},\,\,\,
\Delta_{\alpha\gamma\beta},\,\,\, 
\Delta_{\beta\alpha\gamma},\,\,\,
\Delta_{\beta\gamma\alpha},\,\,\,
\Delta_{\gamma\alpha\beta},\,\,\, 
\Delta_{\gamma\beta\alpha}. 
\end{eqnarray}
\noindent Once again, they are all pedal triangles with respect to six points,
say, 
\begin{eqnarray}\label{pedal-1.1}
P_{\alpha\beta\gamma},\,\,P_{\alpha\gamma\beta},\,\, 
P_{\beta\alpha\gamma},\,\,P_{\beta\gamma\alpha},\,\, 
P_{\gamma\alpha\beta},\,\,P_{\gamma\beta\alpha}.
\end{eqnarray}
Remarkably, even though these six points are obtained via constructions that 
start from starkly different sets of premises, they turn out to be concyclic. 
That is,

\begin{theorem}[THE SIX POINT CIRCLE THEOREM]\label{Th-Intr}
There exists a circle that contains all six points in (\ref{pedal-1.1}). 
\end{theorem}
\noindent This is the main result of the paper (cf. Theorem~\ref{tata3.3} 
for a formal statement). The starting point in its proof is to consider, 
besides the area-minimizing triangle inscribed in $\Delta ABC$ and having 
prescribed angles, the {\it area-maximizing triangle circumscribed to} 
$\Delta ABC$, having the same angles. This dual aspect of the problem 
under consideration is very useful in understanding the nature of the 
extremal triangles of a prescribed shape. A concrete step in this regard
is proving a structure theorem, detailing how these three triangles 
are related to one another. For instance, we show that 
{\it the area of the original triangle is the geometric mean of the 
area-minimizer and the area-maximizer triangles}; cf. Theorem~\ref{tata444}. 
Since the conclusion in Theorem~\ref{Th-Intr} involves a circle, 
it is not unnatural that another important tool in its proof 
is the {\it geometrical inversion}. 

In $\S 2$, the properties of the area-minimizing and area-maximizing 
triangles having prescribed angles and being, respectively, inscribed 
and circumscribed in a given triangle are examined. Finally, $\S 3$ is 
devoted to presenting the proof of The Six Point Circle Theorem. 
The presentation in this paper is largely self-contained. For more basic 
concepts and definitions related to the geometry of the triangle the 
interested reader is referred to, e.g. \cite{Ki}, and the references therein.

\section{Extremal Triangles of a Prescribed Shape}

\noindent Throughout, we let $\Delta ABC$ denote the triangle with 
vertices $A,B,C$ and denote by $|\Delta ABC|$ the area of $\Delta\,ABC$. 

Given a triangle $A_1A_2A_3$ along with a triangle $B_1B_2B_3$ inscribed 
in it, we first recall a result from \cite{AM} that gives 
a procedure for obtaining a triangle, $C_1C_2C_3$, that is inscribed in 
$\Delta B_1B_2B_3$ and is homotopic to $\Delta A_1A_2A_3$. By definition, 
two triangles are homotopic if their corresponding sides are parallel. 
The concurrence point of the lines passing through the corresponding vertices 
of two homotopic triangles is the homotopy center. The justification of the 
following useful result is left to the interested reader (a proof 
is given in \cite{AM}).
\begin{proposition}\label{tata111}
Let $\Delta A_1A_2A_3$ be arbitrary and assume that $B_1\in A_2A_3$,
$B_3\in A_2A_1$, $B_2\in A_1A_3$ (see Figure~1 below). Take $C_1\in B_2B_3$, 
$C_2\in B_1B_3$, $C_3\in B_1B_2$ such that
\begin{eqnarray}\label{Def-Kk}
\frac{A_1B_3}{B_3A_2}=\frac{B_2C_3}{C_3B_1},\quad
\frac{A_2B_1}{B_1A_3}=\frac{B_3C_1}{C_1B_2},\quad
\frac{A_3B_2}{B_2A_1}=\frac{B_1C_2}{C_2B_3},
\end{eqnarray}
\noindent Then $\Delta A_1A_2A_3$ and $\Delta C_1C_2C_3$ are
homotopic and, in addition, $|\Delta B_1B_2B_3|$ is the geometric 
mean of $|\Delta A_1A_2A_3|$ and $|\Delta C_1C_2C_3|$, i.e.
\begin{eqnarray}\label{SS1}
|\Delta B_1B_2B_3|^2=|\Delta A_1A_2A_3|\cdot|\Delta C_1C_2C_3|. 
\end{eqnarray}
\noindent Conversely, if $\Delta A_1A_2A_3$ and $B_1\in A_2A_3$, 
$B_3\in A_2A_1$, $B_2\in A_1A_3$ are given and $C_1\in B_2B_3$, $C_2\in B_1B_3$, 
$C_3\in B_1B_2$ are such that $\Delta A_1A_2A_3$ and $\Delta C_1C_2C_3$ 
are homotopic, then (\ref{Def-Kk}) and (\ref{SS1}) hold. 
\end{proposition}

\begin{center}
\includegraphics[scale=0.85]{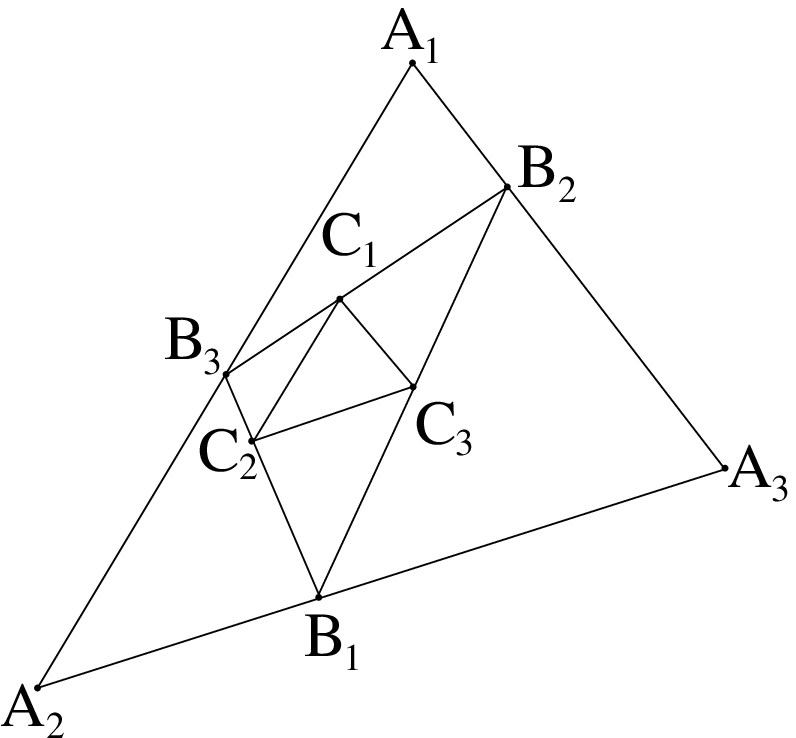}
\end{center}
\centerline{{\bf Figure~1}}
\vskip 0.15 in

\begin{definition}\label{Def-Fun}
(i) Call a triangle $XYZ$ {\tt inscribed} in a given triangle 
$ABC$ if each of the lines $AB$, $BC$ and $AC$ contain precisely 
one of the vertices $X$, $Y$, $Z$. Call $\Delta XYZ$ {\tt circumscribed} 
to $\Delta ABC$ if $\Delta ABC$ is inscribed in $\Delta XYZ$. 

(ii) Given two fixed triangles $ABC$, called of {\tt reference},
and $MNP$, called {\tt fundamental}, let ${\mathcal I}$ be the set
of all triangles which are similar to the fundamental one and are
inscribed in the triangle of reference, and let ${\mathcal C}$
be the set of all triangles similar to the fundamental one
and circumscribed to the reference one.

(iii) Given $\Delta\,ABC$, triangle of reference, and a fundamental 
triangle with angles $\alpha,\beta,\gamma$, denote by 
${\mathcal I}_{\alpha\beta\gamma}$ the set 
of all triangles $XYZ$ with $X,Y,Z$ belonging, respectively, to 
the lines $BC$, $AC$, $AB$, and for which  
$\aangle\,X$, $\aangle\,Y$, $\aangle\,Z$ coincide with 
$\alpha$, $\beta$, $\gamma$ (in this order). 

Likewise, let ${\mathcal C}_{\alpha\beta\gamma}$ denote the set 
of all triangles $XYZ$ for which $A,B,C$ belonging, respectively, to 
the lines $YZ$, $XZ$, $XY$, and for which  
$\aangle\,X$, $\aangle\,Y$, $\aangle\,Z$ coincide with 
$\alpha$, $\beta$, $\gamma$ (in this order). 
\end{definition}

Clearly, 

\begin{eqnarray}\label{IIC}
\begin{array}{l}
{\mathcal{I}}={\mathcal{I}}_{\alpha\beta\gamma}\bigcup
\,{\mathcal{I}}_{\alpha\gamma\beta}\bigcup
\,{\mathcal{I}}_{\beta\alpha\gamma}\bigcup
\,{\mathcal{I}}_{\beta\gamma\alpha}\bigcup
\,{\mathcal{I}}_{\gamma\alpha\beta}\bigcup
\,{\mathcal{I}}_{\gamma\beta\alpha},
\\[8pt]
{\mathcal{C}}={\mathcal{C}}_{\alpha\beta\gamma}\bigcup
\,{\mathcal{C}}_{\alpha\gamma\beta}\bigcup
\,{\mathcal{C}}_{\beta\alpha\gamma}\bigcup
\,{\mathcal{C}}_{\beta\gamma\alpha}\bigcup
\,{\mathcal{C}}_{\gamma\alpha\beta}\bigcup
\,{\mathcal{C}}_{\gamma\beta\alpha}.
\end{array}
\end{eqnarray}

Our next result shows the existence of area-minimizers and
area-maximizers in the above classes. The issue of uniqueness 
for these extremal triangles will be dealt with in Theorem~\ref{tata444} 
below. 

\begin{theorem}\label{tata333} 
There exists a triangle 
$\Delta\,M_1N_1P_1\in {\mathcal I}_{\alpha\beta\gamma}$ which has the 
smallest area amongst all the triangles in the class
${\mathcal I}_{\alpha\beta\gamma}$. Also, there exists a triangle 
$\Delta\,M_2N_2P_2\in {\mathcal C}_{\alpha\beta\gamma}$ which has the 
largest area amongst all the triangles in the class
${\mathcal C}_{\alpha\beta\gamma}$. In addition, any two such 
triangles satisfy 
\begin{eqnarray}\label{As-Qw}
|\Delta M_1N_1P_1|\cdot |\Delta M_2N_2P_2|=|\Delta ABC|^2.
\end{eqnarray}
 
In fact, similar results are valid for each of the classes appearing in 
(\ref{IIC}), including ${\mathcal{I}}$ and ${\mathcal{C}}$. 
\end{theorem}

\begin{proof} It suffices to only deal with the class 
${\mathcal I}_{\alpha\beta\gamma}$ since a similar reasoning
applies to all the classes in (\ref{IIC}). 
The first issue is to show that there exists at least one 
area-minimizer in ${\mathcal I}_{\alpha\beta\gamma}$. 
We briefly sketch the argument for this, which is based on results from
Calculus and Analytical Geometry. To get started, fix 
a unit vector $\vec{w}=w_1\vec{i}+w_2\vec{j}$ and consider the
class of all triangles $XYZ$ satisfying 
\begin{eqnarray}\label{1.const}
\aangle\,X=\alpha,\quad\aangle\,Y=\beta,\quad 
\aangle\,Z=\gamma,\quad X\in BC,\quad Z\in AB,\quad XZ\,\|\,\vec{w}.    
\end{eqnarray}
\noindent It is then clear that the coordinates of $Y$ and $Z$ depend 
linearly on those of $X$, with coefficients varying continuously with 
$w_1,w_2$. Since $X$ moves on a line (i.e., $BC$), it follows that the 
geometrical locus of the vertex $Y$ is itself a line. 
In addition, the characteristics of this line, called $L$, depend 
continuously on $w_1$, $w_2$. Now, the position of $X$ which, under the 
constraints (\ref{1.const}), yields a triangle $XYZ$ in the class 
${\mathcal{I}}_{\alpha\beta\gamma}$, is the one corresponding to the case 
when $Y=L\cap AC$. Linear algebra considerations then show that the 
coordinates of $X$ in this critical case are continuously dependent 
on $w_1,w_2$. Using the formula which expresses the area of 
a triangle in terms of the coordinates of its vertices, we then deduce that 
$|\Delta XYZ|$ -corresponding the case when 
$\Delta XYZ\in{\mathcal{I}}_{\alpha\beta\gamma}$- is a continuous 
function in $w_1$, $w_2$. Since $w_1$ and $w_2$ vary in a compact set, 
it follows that the assignment $(w_1,w_2)\mapsto |\Delta XYZ|$ is 
bounded and it attains its extrema. In particular, there exists an 
area-minimizing triangle in ${\mathcal{I}}_{\alpha\beta\gamma}$. 

To prove that there exists (at least) one area-maximizing triangle
in ${\mathcal{C}}_{\alpha\beta\gamma}$, we proceed as follows. 
Starting with an area-minimizing triangle $\Delta M_1N_1P_1$ 
(whose existence has just been established), construct a 
a new triangle, $\Delta M_2N_2P_2$ by taking the parallels through 
$A$, $B$ and $C$ to the sides of $\Delta\,M_1N_1P_1$.  
As a result, $\Delta M_2N_2P_2$ is homotopic to $\Delta\,M_1N_1P_1$ 
and circumscribed to $\Delta\,ABC$. Hence, 
$\Delta\,M_2N_2P_2\in{\mathcal C}_{\alpha\beta\gamma}$ and, 
thanks to Proposition~\ref{tata111}, 
\begin{eqnarray}\label{Min=H}
|\Delta M_2N_2P_2|\cdot|\Delta M_1N_1P_1|=|\Delta ABC|^2.
\end{eqnarray}
\noindent We claim that $\Delta M_2N_2P_2$ is area-maximizing in the
class ${\mathcal{C}}_{\alpha\beta\gamma}$. Indeed, if 
$\Delta M''N''P''\in{\mathcal{C}}_{\alpha\beta\gamma}$ is arbitrary, 
let $M'\in BA$, $N'\in BC$, $P'\in AC$ (cf. Figure~2) be such that
\begin{eqnarray}\label{1.1-Kj}
\frac{AM'}{M'B}=\frac{P''C}{CN''},\quad\frac{BN'}{N'C}=\frac{M''A}{AP''},
\quad\frac{CP'}{P'A}=\frac{N''B}{BM''}.
\end{eqnarray}
\noindent Proposition~\ref{tata111} gives that $\Delta\,M'N'P'$ and 
$\Delta\,M''N''P''$ are homotopic 
(thus $\Delta\,M'N'P'\in{\mathcal{I}}_{\alpha\beta\gamma}$) and 
\begin{eqnarray}\label{1.LK-4}
|\Delta M''N''P''|\cdot |\Delta M'N'P'|=|\Delta ABC|^2.
\end{eqnarray}
\noindent Since by definition $|\Delta M'N'P'|\geq |\Delta M_1N_1P_1|$, 
(\ref{Min=H}), (\ref{1.LK-4}) imply that 
$|\Delta M''N''P''|\leq |\Delta M_2N_2P_2|$. This justifies our claim
that $\Delta M_2N_2P_2$ is area-maximizing in the
class ${\mathcal{C}}_{\alpha\beta\gamma}$.

Going further, we note that proving (\ref{As-Qw}) uses a similar 
circle of ideas. Concretely, let $M'\in BA$, $N'\in BC$, $P'\in AC$ 
(cf. Figure~2) be such that

\begin{eqnarray}\label{1.1}
\frac{AM'}{M'B}=\frac{P_2C}{CN_2},\quad\frac{BN'}{N'C}=\frac{M_2A}{AP_2},
\quad\frac{CP'}{P'A}=\frac{N_2B}{BM_2}.
\end{eqnarray}
\noindent Proposition~\ref{tata111} shows that $\Delta\,M'N'P'$ and 
$\Delta\,M_2N_2P_2$ are homotopic and 
\begin{eqnarray}\label{1.1bis}
|\Delta M_2N_2P_2|\cdot |\Delta M'N'P'|=|\Delta ABC|^2.
\end{eqnarray}
\noindent The former property also entails 
$\Delta\,M'N'P'\in{\mathcal{I}}_{\alpha\beta\gamma}$ 
and, hence, 
\begin{eqnarray}\label{1.1bV}
|\Delta M_1N_1P_1|\leq |\Delta M'N'P'|. 
\end{eqnarray}
\begin{center}
\includegraphics[scale=1.0]{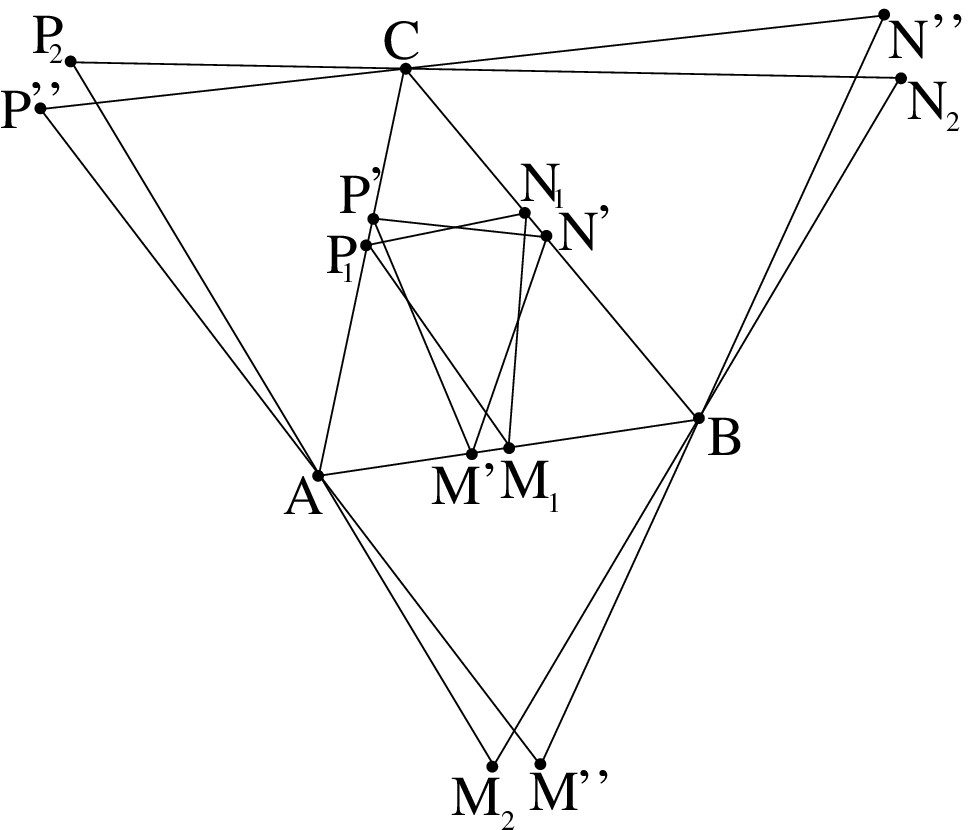}
\end{center}
\centerline{{\bf Figure~2}}

\vskip 0.15in

The parallels through $A$, $B$ and $C$ to the sides of the 
triangle $\Delta\,M_1N_1P_1$ determine the triangle $M''N''P''$ which is 
homotopic to $\Delta\,M_1N_1P_1$ and circumscribed to $\Delta\,ABC$. 
As a result, $\Delta\,M''N''P''\in{\mathcal C}_{\alpha\beta\gamma}$
which forces 
\begin{eqnarray}\label{1.2}
|\Delta M''N''P''|\leq |\Delta M_2N_2P_2|.
\end{eqnarray}
\noindent In this context, Proposition~\ref{tata111} also gives that 
\begin{eqnarray}\label{1.2bis}
|\Delta M''N''P''|\cdot |\Delta M_1N_1P_1|=|\Delta ABC|^2.
\end{eqnarray}

By combining (\ref{1.1bis})-(\ref{1.2bis}) we obtain
\begin{eqnarray}\label{1.3}
|\Delta ABC|^2 &=& |\Delta M_1N_1P_1||\Delta M''N''P''|
\nonumber\\[4pt]
&\leq & |\Delta M_2N_2P_2||\Delta M'N'P'|=|\Delta ABC|^2, 
\end{eqnarray}
\noindent so that, we actually have
$|\Delta M''N''P''|=|\Delta M_2N_2P_2|,\quad
|\Delta M_1N_1P_1|=|\Delta M'N'P'|$. Therefore,
$|\Delta ABC|^2=|\Delta M_1N_1P_1||\Delta M''N''P''|
=|\Delta M_1N_1P_1||\Delta M_2N_2P_2|$, 
finishing the proof of (\ref{As-Qw}).  
\end{proof}

\noindent{\bf Remark~1.} {\it It is implicit in the proof of the above 
theorem that if $\Delta M_1N_1P_1$ is an area-minimizer in the class
${\mathcal{I}}_{\alpha\beta\gamma}$, then the triangle $M_2N_2P_2$ that 
is in ${\mathcal{C}}_{\alpha\beta\gamma}$ and is homotopic to 
$\Delta M_1N_1P_1$ (constructed as in Proposition~\ref{tata111}) 
is an area-maximizer in the class ${\mathcal{C}}_{\alpha\beta\gamma}$.
Conversely, if $\Delta M_2N_2P_2$ is an area-maximizer in the class
${\mathcal{C}}_{\alpha\beta\gamma}$, then the triangle $M_1N_1P_1$ that 
is in ${\mathcal{I}}_{\alpha\beta\gamma}$ and is homotopic to 
$\Delta M_2N_2P_2$ (again, constructed as in Proposition~\ref{tata111}) 
is an area-minimizer in the class ${\mathcal{I}}_{\alpha\beta\gamma}$.}

\vskip 0.10in

The main result in this section describes the structure of the 
configuration consisting of a given triangle $ABC$ along with the 
area-minimizing triangle in ${\mathcal{I}}_{\alpha\beta\gamma}$ 
and the area-maximizing triangle in ${\mathcal{C}}_{\alpha\beta\gamma}$.
Recall that an antipedal triangle of a point $P$ with respect to $\Delta ABC$ 
is the triangle whose vertices are the intersection points of the 
perpendiculars to $PA$, $PB$, and $PC$, passing through $A$, $B$, and $C$, 
respectively. Also, the isogonal of a point $P$ in $\Delta ABC$ is the 
concurrence point of the reflections of the lines $PA$, $PB$, and $PC$ 
across the angle bisectors of $A$, $B$, and $C$. 

\begin{theorem}\label{tata444}
Retain notation and conventions introduced earlier. 
Given $\Delta ABC$, reference triangle, and $\Delta MNP$ fundamental triangle
with angles $\alpha,\beta,\gamma\in(0,\pi)$, let 
${\mathcal{I}}_{\alpha\beta\gamma}$ and
${\mathcal{C}}_{\alpha\beta\gamma}$ be as before. 

Then there exist a unique $\Delta\,EFD\in {\mathcal{I}}_{\alpha\beta\gamma}$ 
which is area-minimizing in this class, along with a unique 
$\Delta\,QRS\in {\mathcal{C}}_{\alpha\beta\gamma}$ which is 
area-maximizing in this class. In addition, these triangles enjoy 
the following properties:
\begin{enumerate}
\item[(i)] The triangles $EFD$ and $QRS$ are homotopic;
\item[(ii)] The triangles $EFD$ and $QRS$, are the pedal and the 
antipedal triangles of some points $K$ and $L$, respectively, 
with respect to $\Delta\,ABC$;
\item[(iii)] Let $T$ be the concurrence point of the projections of 
$Q$, $R$, $S$ on the sides of the triangle $ABC$. Then the pairs of points
$(L,K)$ and $(T,L)$ are isogonal in the triangles $ABC$ and $QRS$, 
respectively;
\item[(iv)] If $O$ is the homotopy center of 
$\Delta\,EFD$ and $\Delta\,QRS$, then $O$, $K$, $T$ lie on the same line.
\end{enumerate}
\end{theorem}

Before proceeding with the proof of this theorem, a comment is in order. 

\vskip 0.08in
\noindent{\bf Remark~2.} {\it As a corollary of the above theorem, there exists 
a unique area-minimizing triangle in each of the classes
${\mathcal{I}}$, ${\mathcal{I}}_{\alpha\beta\gamma}$, ..., listed in 
the first line of (\ref{IIC}). Moreover, there exists a unique area-maximizing 
triangle in each of the classes ${\mathcal{C}}$, 
${\mathcal{C}}_{\alpha\beta\gamma}$, ...., listed in 
the second line of (\ref{IIC}). In each case, similar properties to 
$(i)-(iv)$ above hold.}

\begin{center}
\includegraphics[scale=1.0]{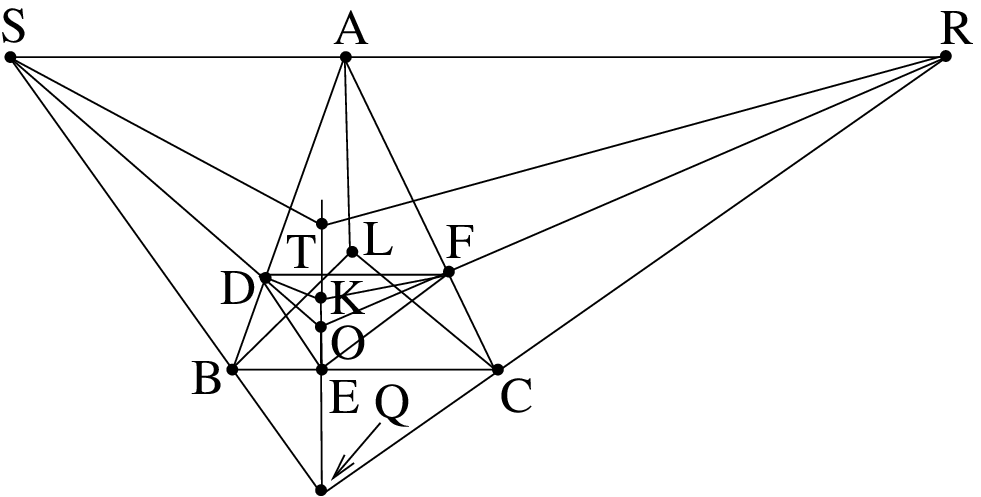}
\end{center}
\vskip 0.1in
\centerline{\bf Figure~3}
\vskip 0.15 in

\noindent{\it Proof of Theorem~\ref{tata444}.} By design, the triangles $QRS$ 
and $MNP$ are similar (cf. Figure~3), so that $\aangle\,{Q}=\aangle\,{M}$, 
$\aangle\,{R}=\aangle\,{N}$, $\aangle\,{S}=\aangle\,{P}$.   
If we do not impose the condition that the triangle $QRS$ has maximal area
in the class ${\mathcal{C}}_{\alpha\beta\gamma}$, then the triangle $QRS$ 
is not completely determined. There are, however, certain restrictions 
on the location of its vertices, due solely to the membership to 
${\mathcal{C}}_{\alpha\beta\gamma}$. Concretely, the point $S$ must lie 
on the circle passing through $A$ and $B$ and such that the arc 
$\overset{\frown}{AB}=2\aangle\,{S}$. Let us denote by $O_1$ the 
center of this circle (see Figure~4). Likewise, the point $Q$ lies on a circle 
(whose center we denote by $O_2$) passing through $B$ and $C$ 
and such that $\overset{\frown}{BC}=2\aangle\,{Q}$. Finally, the point 
$R$ lies on a circle (centered at some point, denoted by $O_3$) passing 
through $A$ and $C$ and for which $\overset{\frown}{AC}=2\aangle\,{R}$.
  
Since $\aangle\,{M}+\aangle\,{N}+\aangle\,{P}
=\aangle\,{Q}+\aangle\,{R}+\aangle\,{S}=\pi$, it is not difficult 
to see that these three circles, have a common point, which we denote by $L$.
Observe next that a particular position of the point $S$ on the circle $(O_1)$
determines also the positions of the points $Q$ and $R$ on the circles
$(O_2)$ and $(O_3)$, respectively. The idea is now to determine the 
position of the point $S$ on the circle $(O_1)$ for which the triangle 
$QRS$ has maximal area.

\begin{center}
\includegraphics[scale=0.7]{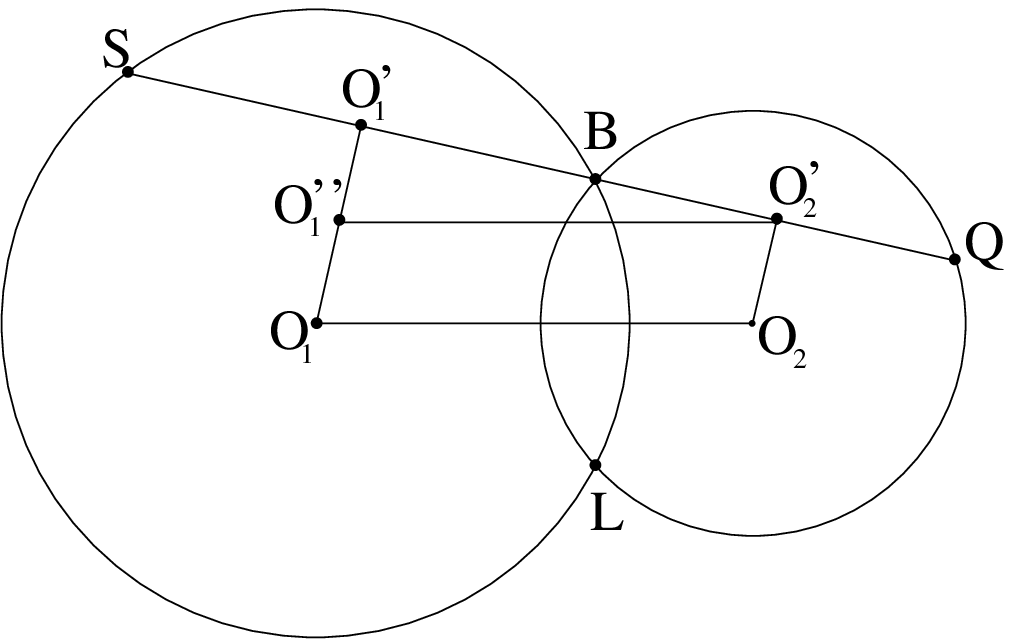}
\end{center}
\centerline{{\bf Figure~4}}
\vskip 0.15 in

Since the angles of $\Delta QRS$ are fixed, maximizing $|\Delta QRS|$ 
amounts to maximizing the length of one of its sides, say $SQ$. 
Let $O'_1$ and $O'_2$ be, respectively, the projections of $O_1$ and 
$O_2$ onto the line $SQ$, and let $O''_1\in O'_1O_1$ be such that 
$O''_1O'_2\parallel O_1O_2$. Then we can write 
\begin{eqnarray}\label{boise1}
O_1O_2=O''_1O'_2\geq O'_1O'_2=BO'_1+BO'_2
=\frac{BS}{2}+\frac{BQ}{2}=\frac{SQ}{2}.
\end{eqnarray}
\noindent This implies $2O_1O_2\geq SQ$, an inequality from which we see
that the maximum of $|\Delta QRS|$ is attained when $SQ\parallel O_1O_2$. 

Since the location of the points $O_1$, $O_2$, $O_3$ depends only 
on $\Delta ABC$ and angles $\alpha,\beta,\gamma$, the result just 
established shows that there exists a {\it unique} area-maximizing 
triangle in ${\mathcal{C}}_{\alpha\beta\gamma}$. Given an area-minimizing
triangle $\Delta DEF$ in the class ${\mathcal{I}}_{\alpha\beta\gamma}$,
we can associate with it a homotopic triangle $\Delta QRS$ in 
${\mathcal{C}}_{\alpha\beta\gamma}$, which can be constructed following 
the recipe given in Proposition~\ref{tata111}. Using the Remark~1, 
the latter triangle is area-maximizing in ${\mathcal{C}}_{\alpha\beta\gamma}$ 
and, hence, must coincide with the unique area-maximizing triangle described 
in the first part of the current proof. This shows that $\Delta DEF$ 
is itself uniquely determined by the property that it is area-minimizing 
in ${\mathcal{I}}_{\alpha\beta\gamma}$. This concludes the proof of 
the first claim made in the statement of Theorem~\ref{tata444}. 
The above considerations also show that $\Delta DEF$ and $\Delta QRS$ 
are homotopic, thus proving item $(i)$. 

As regards $(ii)$, let us recall from the first part of the proof 
that the triangle $QRS$ has a maximal area when $S=LO_1\cap (O_1)$, 
$Q=LO_2\cap (O_2)$, $R=LO_3\cap (O_3)$. In particular, this implies that  
the triangles $QRS$ and $O_1O_2O_3$ are homotopic with homotopy center 
$L$ and dilation factor $\frac{1}{2}$. Next, since $O_1O_2\parallel SQ$
and $O_1O_2\perp BL$, it follows that $BL\perp SQ$.
Similarly $LC\perp RQ$ and $LA\perp SR$. Altogether, these observations 
imply that the triangle $SRQ$ is the antipedal of $L$ with respect to the 
triangle $ABC$. Now, using that $DF\parallel SR$, $DE\parallel SQ$, and 
$FE\parallel RQ$ we obtain that $AL\perp DF$, $BL\perp DE$, and $CL\perp EF$. 
This means that the projections of the vertices of $\Delta ABC$ on the sides 
of $\Delta DEF$ are concurrent, therefore establishing that the triangles 
$ABC$ and $DEF$ are orthologic. As is well-know, this gives that 
the perpendiculars from $D,E,F$ onto $AB,$ $BC,$ $AC$ respectively, 
are also concurrent in a point, say $K$. Hence, $\Delta DEF$ is the pedal
triangle of $K$ with respect to the triangle $ABC$. This finishes the proof
of $(ii)$.  
 
Consider next $(iii)$. Since the triangles $ABC$ and $SQR$ are orthologic, 
there exists a point $T$ such that $ST\perp AB,$ $QT\perp BC,$ and 
$RT\perp AC$. In the quadrilateral $ADKF$ we have 
\begin{eqnarray}\label{1.5}
\aangle\,{KAF}=\aangle\,{FDK}\quad\mbox{ and }\quad 
\aangle\,{FDK}+\aangle\,{FDA}=\aangle\,{FDA}+\aangle\,{LAD}=\frac{\pi}{2}.
\end{eqnarray}
\noindent Therefore, 
\begin{eqnarray}\label{1.6}
\aangle\,{FDK}=\aangle\,{LAD}.
\end{eqnarray}
\noindent From (\ref{1.5}) and (\ref{1.6}) we see that the lines 
$AL$ and $AK$ are isogonal. Similarly, $BL$ and $BK$ 
are isogonal, proving that $L$ and $K$ are isogonal points 
in $\Delta\,ABC$. This concludes the proof of $(iii)$. 

To prove $(iv)$, we start by observing that the dilation 
of center $O$ and factor $\frac{SQ}{DE}$ (i.e., the dilation which takes
$\Delta DEF$ into $\Delta SQR$) transforms the lines $DK$, $EK$ 
and $FK$ into the lines $ST$, $RT$ and $QT$, respectively. 
(Indeed, this is a direct corollary of the fact that dilations map
parallel lines into parallel lines.) Thus, under such a transformation, 
the point $K$ is mapped into $T$ and, as a consequence, the points $O$, 
$K$ and $T$ are collinear. The proof of Theorem~\ref{tata444}
is now finished. 
\hfill$\Box$
\vskip 0.08in

The result below essentially says the that a pedal triangle 
(relative to $\Delta ABC$) is area-minimizing amongst all 
triangles of a similar shape inscribed in $\Delta ABC$. 
\begin{theorem}\label{New-Th1W}
Let $P$ be a point contained inside the circumcircle of a given triangle $ABC$. 
Let $D,E,F$ be the projections of $P$ onto the sides $BC$, $AC$ and 
$AB$, respectively, and set 
\begin{eqnarray}\label{Yr-W1}
\alpha:=\aangle\,EDF,\quad
\beta:=\aangle\,DEF,\quad
\gamma:=\aangle\,EFD. 
\end{eqnarray}
\noindent Then $\Delta DEF$ is the area-minimizing triangle in the class 
${\mathcal{I}}_{\alpha\beta\gamma}$ (cf. (ii) of Definition~\ref{Def-Fun}). 
\end{theorem}

\noindent{\it Proof.} Let $K$ be the isogonal of $P$ (relative to 
$\Delta ABC$) and denote by $\Delta QRS$ the antipedal of $K$ with 
respect to $\Delta ABC$. Since $\Delta\,DEF$ is the pedal triangle of the 
point $P$ with respect to $\Delta\,ABC$, then quadrilateral
$AFPE$ is inscribable. This implies that 
$\aangle PEF+\aangle PAF=\frac{\pi}{2}$. Next, since $K$ is the 
isogonal of $P$, this further entails 
$\aangle PEF+\aangle KAE=\frac{\pi}{2}$. This means that $AK\bot EF$, i.e. 
$RS\|EF$. Via a similar reasoning, we can show that the two other pairs 
of corresponding sides in $\Delta QRS$ and $\Delta DEF$ are parallel. 
This shows that $\Delta QRS$ and $\Delta DEF$ are homotopic. 
In particular, the angles of $\Delta QRS$ are $\alpha$, $\beta$, $\gamma$. 

Recall the (three, concurrent) circles $(O_1)$, $(O_2)$, $(O_3)$ 
introduced in the proof of Theorem~\ref{tata444}. Then 
$S\in (O_1)$, $R\in (O_2)$, $Q\in (O_3)$. Given that the 
quadrilateral $AKCR$ is inscribable, it follows that $K\in (O_1)$. 
Similarly, $K\in (O_2)$, $K\in (O_3)$. Hence, $K$ is the 
common point of $(O_1)$, $(O_2)$, $(O_3)$, which implies that 
$\Delta QRS$ is homotopic to $\Delta O_1O_2O_3$. As pointed out already 
in the proof of Theorem~\ref{tata444}, this necessarily implies that 
$\Delta QRS$ is a maximizer in the class ${\mathcal{C}}_{\alpha\beta\gamma}$. 
In turn, by $(iii)$ in Theorem~\ref{tata444} this forces 
$\Delta DEF$ to be an area minimizer in ${\mathcal{I}}_{\alpha\beta\gamma}$.
\hfill$\Box$
\vskip 0.08in

We end with a result which appears to be folklore
(see \cite{Ga}; a new proof has been given in \cite{AM}). 
\begin{theorem}\label{tata100}
Assume $ABC$ to be a given triangle and denote by $O$ and $R$ the center 
and the radius of the circumcircle, respectively. 
Let $M$ be an arbitrary point in $\Delta ABC$ and let 
$\Delta\alpha\beta\gamma$ be the pedal triangle of $M$ with respect to 
$\Delta ABC$. Then 
\begin{eqnarray}\label{Oprisan} 
\frac{|\Delta\alpha\beta\gamma|}{|\Delta ABC|}=\frac{|R^2-OM^2|}{4R^2}.
\end{eqnarray}
\end{theorem}

\section{The Main Result}

This section is devoted to stating and proving the main result in this 
paper, Theorem~\ref{tata3.3}. First, we take care of a 
number of prerequisites, starting with the result below. 
\begin{theorem}\label{tata3.1} 
Let $M$ be a point contained in the circumcircle of $\Delta\,ABC$, 
which is assumed to have center $O$ and radius $R$. In addition, let 
$N\in OM$ be such that $OM\cdot ON=R^2$. Then the pedal triangles, 
$\Delta M_1M_2M_3$ and $\Delta N_1N_2N_3$, of the points $M$ and $N$ with 
respect to $\Delta\,ABC$ are similar (with pairs of equal angles having 
vertices located on the same sides of $\Delta\,ABC$).
Furthermore, 
\begin{eqnarray}\label{Ar-1GT}
|\Delta M_1M_2M_3|\leq |\Delta N_1N_2N_3|. 
\end{eqnarray}
Conversely, if $M$ and $N\in OM$ are two points such that their pedal 
triangles with respect to $\Delta\,ABC$ are similar 
(with pairs of equal angles having vertices on the same sides 
of $\Delta\,ABC$), then $OM\cdot ON=R^2$. 
\end{theorem}

\begin{proof} Let $X,$ $Y$ and $Z$ be the centers of the circles $AMB$, 
$BMC$ and $AMC$, respectively (see Figure~5). By an inversion of center 
$O$ and modulus $R$, the triangle $ABC$ and the circle $(O)$ remain unchanged. 
The circles $(X)$, $(Y)$ and $(Z)$ are transformed in the circles $(X')$, 
$(Y')$ and $(Z')$, respectively, passing through $A$ and $B$, $B$ and $C$, $A$ 
and $C$, respectively, and have a common point $N$, which is the 
inverse transform of $M$ (see Figure~6).

\begin{center}
\includegraphics[scale=1.0]{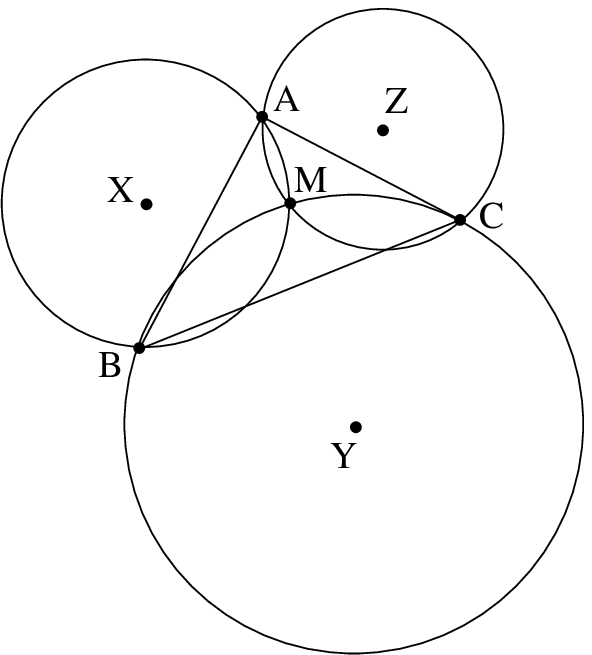}
\end{center}
\centerline{{\bf Figure~5}}
\vskip 0.10in

\begin{center}
\includegraphics[scale=1.0]{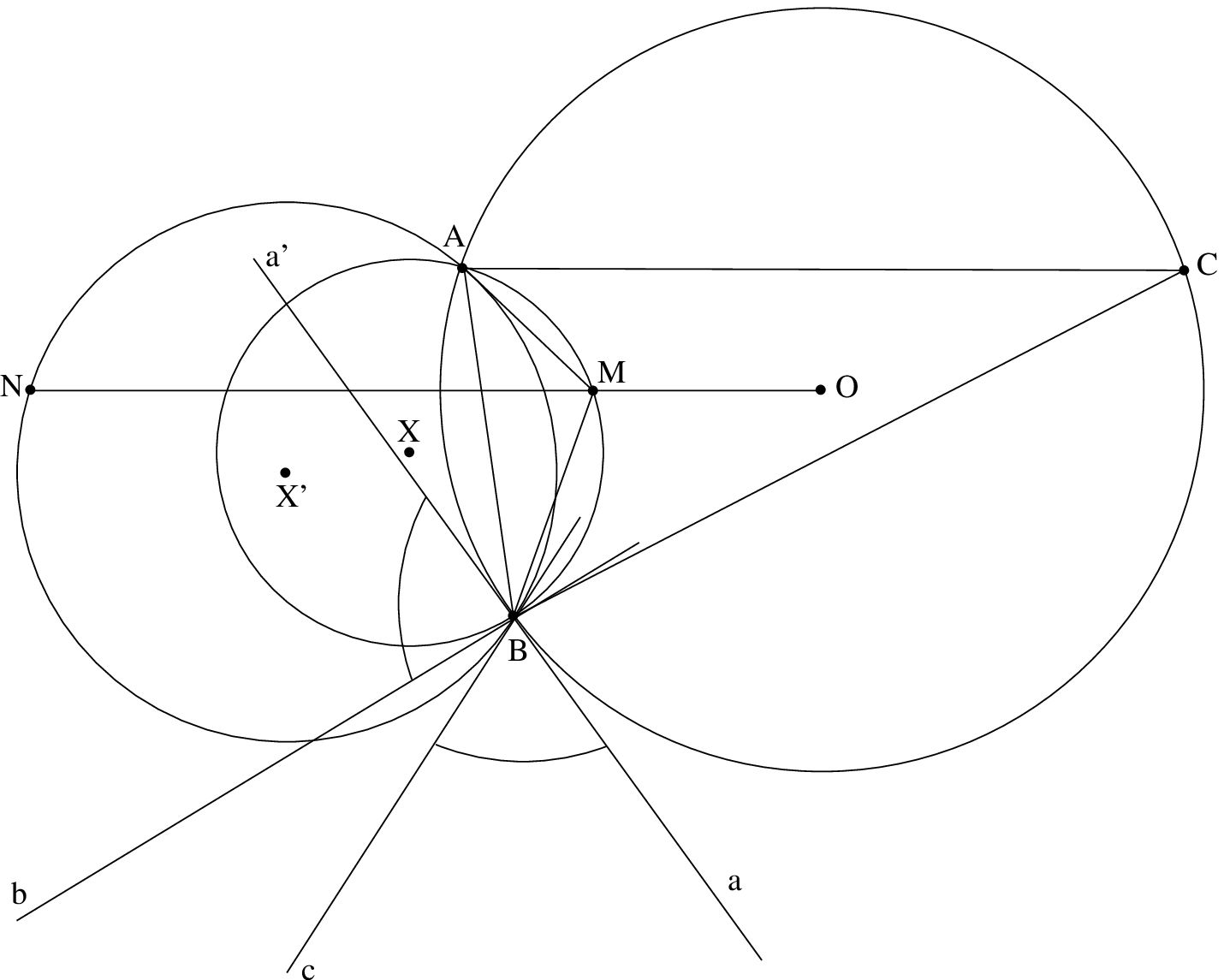}
\end{center}
\centerline{{\bf Figure~6}}
\vskip 0.10in

\begin{center}
\includegraphics[scale=1.0]{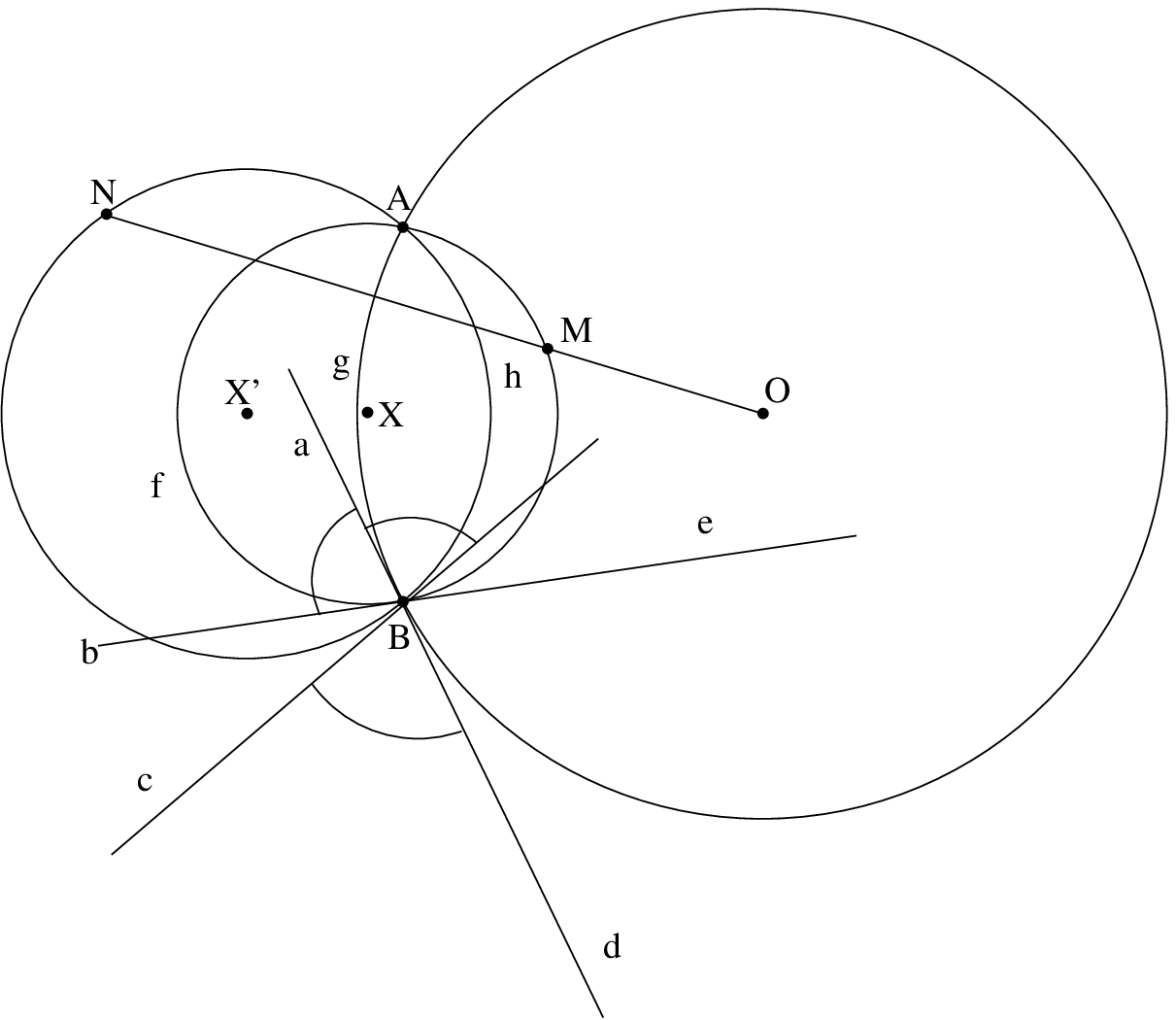}
\end{center}
\centerline{{\bf Figure~7}}
\vskip 0.15in

Consider the lines $aB$, $bB$ and $cB$, which are tangent to the circles 
$(O)$, $(X)$ and $(X')$, respectively (cf. Figure~7). As the inversion 
preserves the angle between two curves, we obtain that 
$\aangle\,((O),(X))=\aangle\,((O),(X'))$, i.e., 
$\aangle\,(bBa)=\aangle\,(cBa)$. We can write $\aangle\,bBa$ as the 
difference of $\aangle\,ABb$ and $\aangle\,ABa$.  
This difference in turn can be re-written as 
$\frac{\overset{\frown}{f}}{2}-\frac{\overset{\frown}{g}}{2}$.  
We have $\aangle\,cBd$ equal to $\aangle\,aBe$, which is further equal to 
$\aangle\,ABe+\aangle\,aBA
=\frac{\overset{\frown}{h}}{2}+\frac{\overset{\frown}{g}}{2}$.  
Using the fact that $\aangle\,bBa=\aangle\,cBd$, we obtain that 
\begin{eqnarray}\label{LT3}
\aangle\,BNA=\aangle\,BMA-2\aangle\,ACB. 
\end{eqnarray}

If $M_1M_2M_3$ and $N_1N_2N_3$ are the pedal triangles of $M$ and $N$ 
respectively, with
\begin{eqnarray}\label{M&N}
M_1,N_1\in AB, M_2,N_2\in BC, M_3,N_3\in AC, 
\end{eqnarray}
\noindent then
\begin{eqnarray}\label{FF26}
\aangle\,ACB &=& \pi-\aangle\,AN_3N_2-\aangle\,BN_2N_3\nonumber
\\[4pt]
&=& \pi-\aangle\,AN_3N_1-\aangle\,N_3-\aangle\,BN_2N_1-\aangle\,N_2.
\label{AOE}
\end{eqnarray}
\noindent In addition, the quadrilaterals $N_3NN_1A$ and $N_2NN_1B$ can 
be inscribed in a circle, therefore 
\begin{eqnarray}\label{LT1}
\aangle\,N_1NA=\aangle\,N_1N_3A \mbox{ and } 
\aangle\,N_1N_2B=\aangle\,BNN_1.
\end{eqnarray}
\noindent Furthermore, 
\begin{eqnarray}\label{LT2} 
\aangle\,BNA=\aangle\,N_1NA+\aangle\,BNN_1,\,\,\,\,\mbox{and}\quad
\aangle\,N_1=\pi-\aangle\,N_2-\aangle\,N_3, 
\end{eqnarray}
\noindent Combining (\ref{LT3}), (\ref{AOE}), (\ref{LT1}), and (\ref{LT2}), 
we arrive at the conclusion that 
\begin{eqnarray}\label{LT4}
\aangle\,N_1=\aangle\,BNA+\aangle\,ACB.
\end{eqnarray}
\noindent It is also possible to write
\begin{eqnarray}
\aangle\,BMA &=& \pi-\aangle\,MAM_1-\aangle\,MBM_1,\label{LT6}
\\[4pt]
\aangle\,ACB &=& \pi-\aangle\,CAB-\aangle\,CBA,\label{LT7}
\\[4pt]
\aangle\,CAB &=& \aangle\,MAC+\aangle\,MAM_1,\label{LT8}
\\[4pt]
\aangle\,CBA &=& \aangle\,MBC+\aangle\,MBM_1.\label{LT9}
\end{eqnarray}
\noindent Therefore, by combining (\ref{LT6})-(\ref{LT9}), we obtain
\begin{eqnarray}\label{math1}
\aangle\,BMA-\aangle\,ACB=\aangle\,MAC+\aangle\,MBC.
\end{eqnarray} 
\noindent Since the quadrilaterals $M_3AM_1M$ and $MM_1BM_2$ can be 
inscribed in a circle, 
\begin{eqnarray}\label{math2}
\aangle\,MAC=\aangle\,MM_1M_3\,\mbox{ and }\aangle\,MBC=\aangle\,MM_1M_2.  
\end{eqnarray}
\noindent However, 
\begin{eqnarray}\label{math3}
\aangle\, M_1=\aangle\,MM_1M_3+\aangle\,MM_1M_2
\end{eqnarray}
\noindent so therefore, by (\ref{math1}), (\ref{math2}), and (\ref{math3}),
\begin{eqnarray}\label{math4}
\aangle\,M_1=\aangle\,BMA-\aangle\,ACB.
\end{eqnarray}
\noindent The identities (\ref{math4}) and (\ref{LT4}) 
imply that $\aangle\,N_1=\aangle\,M_1$. A similar reasoning 
can be used to show that $\aangle\,N_2=\aangle\,M_2$, and that 
$\aangle\,N_3=\aangle\,M_3$, proving that 
$\Delta N_1N_2N_3$ and $\Delta M_1M_2M_3$ are similar, with their 
corresponding angles' vertices located on the same sides of the 
triangle $ABC$.

Next we will prove (\ref{Ar-1GT}). By Theorem~\ref{tata100}, 
\begin{eqnarray}\label{Ar-1GT2}
\frac{|\Delta M_1M_2M_3|}{|\Delta N_1N_2N_3|} 
=\frac{R^2-OM^2}{ON^2-R^2}
\end{eqnarray}
\noindent and the last fraction is $\leq 1$, as it can be seen from 
the fact that $OM\leq R\leq ON$ and $OM\cdot ON=R^2$. Hence, (\ref{Ar-1GT}) 
follows. 

Now we will prove the converse statement referred to in Theorem~\ref{tata3.1}.
In order for $\Delta N_1N_2N_3$ and $\Delta M_1M_2M_3$ to 
be similar, it is a necessary condition that 
$\aangle\,BNA=\aangle\,BMA-2\aangle\,ACB$, which is equivalent
to $\aangle\,BNA=\frac{\overset{\frown}{f}}{2}-\frac{\overset{\frown}{g}}{2}$. 
This can only occur when $\aangle\,BNA=\frac{\overset{\frown}{h}}{2}$ which 
implies that $N$ is on the circle $X'$. A similar reasoning can be applied 
to show that $N$ must be on the circles $Y'$ and $Z'$, proving that
$N$ is the inversion of $M$. This shows that the converse statement is also 
valid, completing the proof of Theorem~\ref{tata3.1}. 
\end{proof}

\noindent{\bf Remark~3.} (i) From the reasoning above we see that there exist 
precisely six points in the interior of the circumscribed circle of 
$\Delta\,ABC$ such that their antipedal triangles with respect to 
$\Delta\,ABC$ are similar to a given reference triangle $MNP$. We shall call 
these points {\it the interior points of $\Delta MNP$ with respect 
to $\Delta ABC$}.

(ii) Likewise, there exist precisely six points in the exterior of the
circumcircle of $\Delta\,ABC$ with the same property as above. 
We shall call these points {\it the exterior points of $\Delta MNP$ 
with respect to $\Delta ABC$.}

\vskip 0.08in

After this preamble, we are ready to state and prove our main result: 

\begin{theorem}\label{tata3.3}
The pedal points of the six triangles, 
$\Delta_{\alpha\beta\gamma}, 
\Delta_{\alpha\gamma\beta}, 
\Delta_{\beta\alpha\gamma}, 
\Delta_{\beta\gamma\alpha}, 
\Delta_{\gamma\alpha\beta},   
\Delta_{\gamma\beta\alpha},$ 
that are area-minimizers in the classes
${\mathcal{I}}_{\alpha\beta\gamma}, 
\,{\mathcal{I}}_{\alpha\gamma\beta}, 
\,{\mathcal{I}}_{\beta\alpha\gamma}, 
\,{\mathcal{I}}_{\beta\gamma\alpha}, 
\,{\mathcal{I}}_{\gamma\alpha\beta}, 
\,{\mathcal{I}}_{\gamma\beta\alpha}$, respectively, all lie on a circle.
\end{theorem}

Proving Theorem~\ref{tata3.3} is equivalent to proving the following.
Consider $ABC$, triangle of reference, and $\Delta\,D_0E_0F_0$ a fixed 
fundamental triangle (in the terminology of Definition~\ref{Def-Fun}).

Let $M_i$, $i=1,2,...,6,$ be six points such that, 
for each $i\in \{1,2,...,6,\}$, the pedal triangle of $M_i$ with respect 
to $\Delta\,ABC$ is $\Delta\,D_iE_iF_i$ with
\begin{eqnarray}\label{CASA2}
\aangle\,{D_i}=\aangle\,{D_0},\quad
\aangle\,{E_i}=\aangle\,{E_0},\quad
\aangle\,{F_i}=\aangle\,{F_0},
\end{eqnarray}
\noindent and
\begin{eqnarray}\label{CASA3}
&& D_1,D_2,E_3,E_4,F_5,F_6\in AB,\nonumber
\\[4pt]
&& D_3,D_5,E_1,E_6,F_2,F_4\in BC,\nonumber
\\[4pt]
&& D_4,D_6,E_2,E_5,F_1,F_3\in AC,
\end{eqnarray}
\noindent (see Figure~8). Then the points $M_i$, $i=1,2,...,6,$ 
lie on the same circle.

\begin{center}
\includegraphics[scale=1.0]{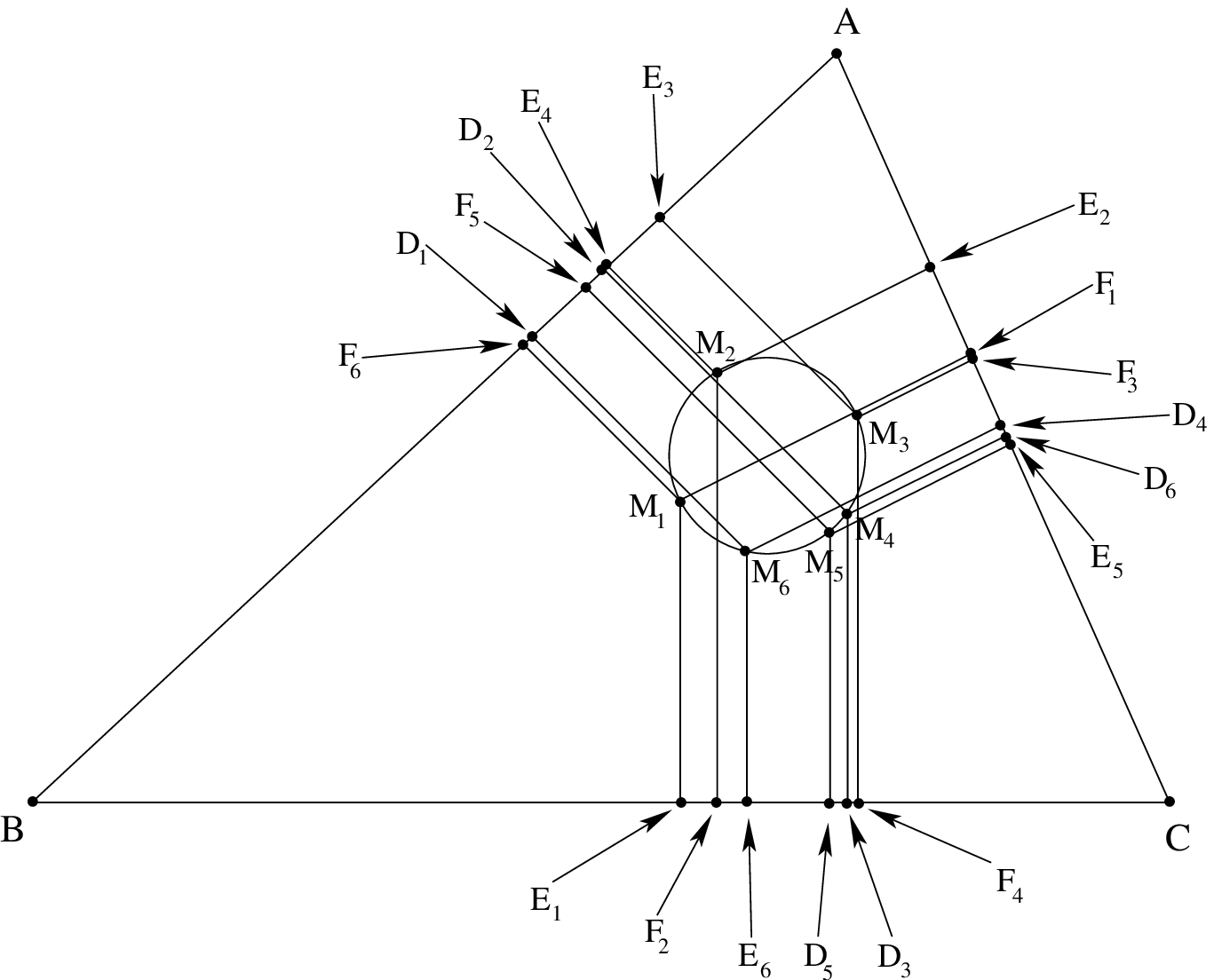}
\end{center}
\centerline{{\bf Figure~8}}
\vskip 0.15in

\begin{proof} Since $\Delta\,D_1E_1F_1$ is the antipedal triangle 
of $M_1$, with respect to $\Delta\,ABC$, we can obtain 
(reasoning similarly to what was done in the proof of Theorem~\ref{tata3.1} 
to show that $\aangle\,M_1=\aangle\,BMA-\aangle\,ACB$) that 
\begin{eqnarray}\label{NbW-3}
\aangle\,{E_1}=\aangle\,{BM_1C}-\aangle\,{A},\quad
\aangle\,{D_1}=\aangle\,{AMB}-\aangle\,{C},\quad
\aangle\,{F_1}=\aangle\,{AM_1C}-\aangle\,{B}.
\end{eqnarray}
\noindent Similarly we obtain $\aangle\,{D_2}=\aangle\,{BM_2A}-\aangle\,{C}$. 
From $\aangle\,{D_0}=\aangle\,{D_1}=\aangle\,{D_2}$, it follows that 
$\aangle\,{BM_1A}=\aangle\,{BM_2A}$, i.e., 
\begin{eqnarray}\label{NbW-5.2}
\begin{array}{l}
\mbox{$M_1,M_2$ lie on a circle $(O_D)$, passing through $A$ and $B$} 
\\[6pt]
\mbox{and with the property that ${\textstyle{\frac{1}{2}}}\overset{\frown}{AB}
=\aangle\,{AM_1B}=\aangle\,C+\aangle\,D_0$.}
\end{array}
\end{eqnarray}
\noindent Likewise (see Figure~9), 
\begin{eqnarray}\label{NbW-5.3}
\begin{array}{l}
\mbox{$M_3,M_4$ lie on a circle $(O_E)$, passing through $A$ and $B$}
\\[4pt]
\mbox{and with the property that 
$\frac{1}{2}\overset{\frown}{AB}=\aangle\,{C}+\aangle\,{E_0}$}, 
\end{array}
\\[8pt]
\begin{array}{l}
\mbox{$M_5,M_6$ lie on a circle $(O_F)$, passing through $A$ and $B$} 
\\[4pt]
\mbox{and with the property that $\frac{1}{2}\overset{\frown}{AB}
=\aangle\,{C}+\aangle\,{F_0}$}, 
\end{array}
\label{NbW-5.4}
\\[8pt]
\begin{array}{l}
\mbox{$M_6,M_4$ lie on a circle $(O'_D)$, passing through $A$ and $C$}
\\[4pt]
\mbox{and with the property that $\frac{1}{2}\overset{\frown}{AC}
=\aangle\,{B}+\aangle\,{D_0}$}, 
\end{array}
\label{NbW-5.5}
\\[8pt]
\begin{array}{l}
\mbox{$M_5,M_2$ lie on a circle $(O'_E)$, passing through $A$ and $C$}
\\[4pt]
\mbox{and with the property that $\frac{1}{2}\overset{\frown}{AC}
=\aangle\,{B}+\aangle\,{E_0}$}, 
\end{array}
\label{NbW-5.6}
\\[8pt]
\begin{array}{l}
\mbox{$M_1,M_3$ lie on a circle $(O'_F)$, passing through $A$ and $C$}
\\[4pt]
\mbox{and with the property that $\frac{1}{2}\overset{\frown}{AC}
=\aangle\,{B}+\aangle\,{F_0}$}. 
\end{array}
\label{NbW-5.7}
\end{eqnarray}

The key step in the proof is to make at this stage an inversion of center 
$A$ and arbitrary modulus. Since the lines $AB$ and $AC$ pass through $A$, 
they will remain unchanged. The points $B$, $C$ become $B'$ and $C'$, 
respectively. The circles $(O_D)$, $(O_E)$, $(O_F)$ will be transformed 
into the lines $d_D$, $d_E$ and $d_F$ respectively, which are concurrent 
at $B'$. Meanwhile, the circles $(O'_D)$, $(O'_E)$ and $(O'_F)$ become 
the lines $d'_D$, $d'_E$ and $d'_F$, respectively, which are concurrent
at $C'$. The points $M_1,\,M_2,\,M_3,\,M_4,\,M_5,\,M_6$ will be transformed 
into $M'_1,\,M'_2,\,M'_3,\,M'_4,\,M'_5$ and $M'_6$,respectively, and 
$M'_1,M'_2\in d_D$, $M'_3,M'_4\in d_E$, $M'_5,M'_6\in d_F$, 
$M'_6,M'_4\in d'_D$, $M'_5,M'_2\in d'_E$, $M'_1,M'_3\in d'_F$.

\begin{center}
\includegraphics[scale=1.0]{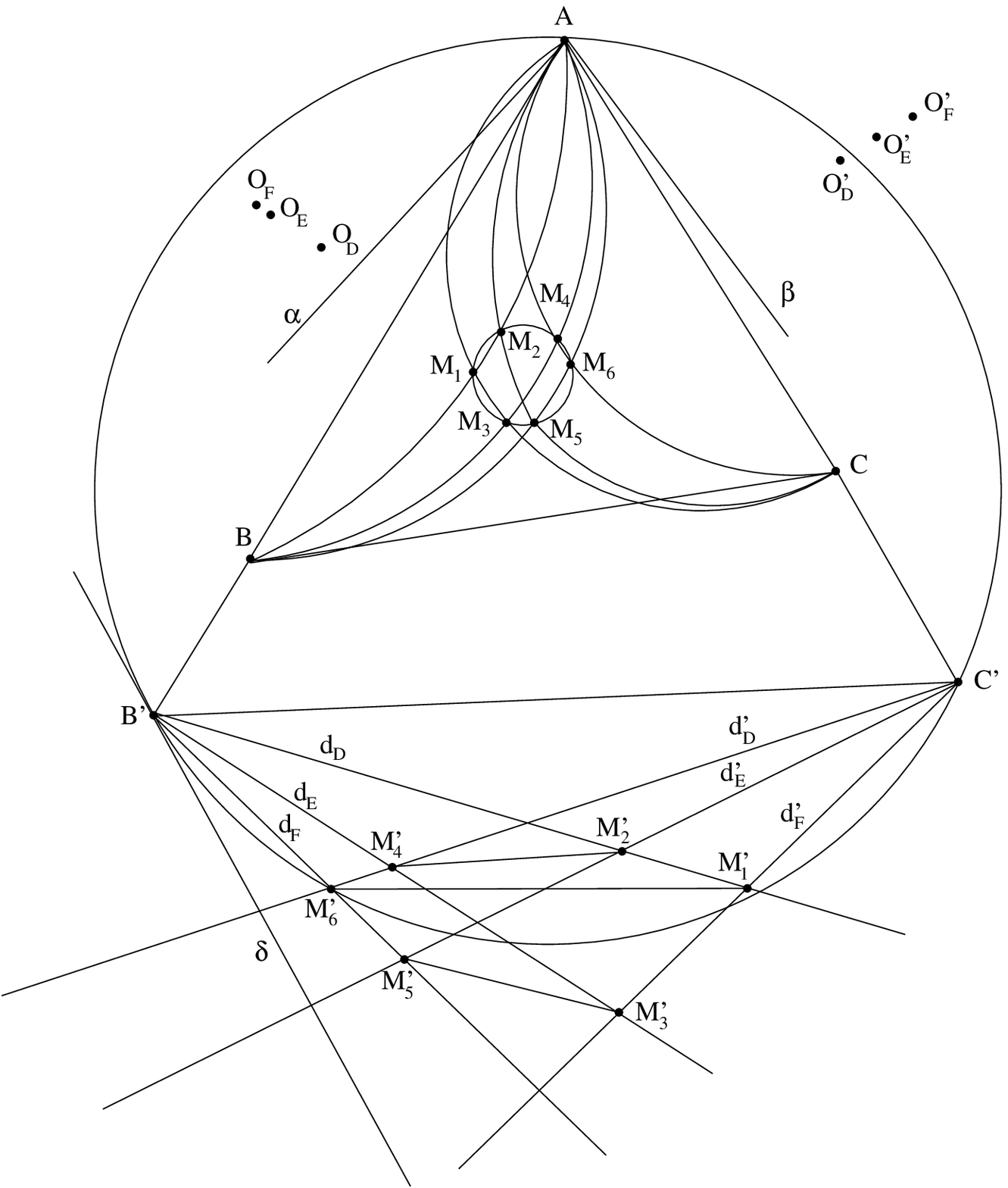}
\end{center}
\centerline{{\bf Figure~9}}
\vskip 0.15in

The goal is to show that the points $M'_1,\,M'_2,\,M'_3,\,M'_4,\,M'_5,\,M'_6$
lie on the same circle. We note that since under 
an inversion the angle between two curves is preserved, we have 
$\aangle\,(d_D,d'_F)=\aangle\,((O_D),(O'_F))$, as the angle 
between two circles is defined to be the angle between the tangents 
at one of their common points. Therefore, 
$\aangle\,(d_D,d'_F)=\aangle\,(\alpha A\beta)=\pi-\aangle\,(O_DAO'_F)$. However,
\begin{eqnarray}\label{ManU4}
\aangle\,(O_DAO'_F)=\aangle\,(O_DAB)+\aangle\,A+\aangle\,(CAO'_F).
\end{eqnarray}
\noindent On the other hand $\aangle\,(O_DAB)=\frac{\pi-\aangle\,(AO_DB)}{2}
=\aangle\,{C}+\aangle\,{D_0}-\frac{\pi}{2}$ and 
$\aangle\,(CAO'_F)=\aangle\,{F_0}+\aangle\,{B}-\frac{\pi}{2}$. Hence,
\begin{eqnarray}\label{Chelsea-1}
\aangle\,(O_DAO'_F)=(\aangle\,{A}+\aangle\,{B}+\aangle\,{C})
+\aangle\,{D_0}+\aangle\,{F_0}-\pi=\aangle\,{D_0}+\aangle\,{F_0},
\end{eqnarray}
\noindent and $\aangle\,(d_D,d'_F)=\pi -\aangle\,(O_DAO'_F)=\aangle\,{E_0}$.

Let $B'{\delta}$ be the tangent at $B$ to the circle determined by the points
$A'$, $B'$, $C'$. Then 
$\aangle\,(d_E,B'C')=\aangle\,({\delta}B'C')-\aangle\,({\delta}B',d_E)$.
However, $\aangle\,({\delta}B'C')=\aangle\, A$, and 
$\aangle\,({\delta}B',d_E)=\aangle\,(C(A,B',C'),d_E)=\aangle\,(BC,(O_E))
=-\aangle\, B+(\pi -\aangle\,{C}-\aangle\,{E_0})$, where $C(A,B,C)$ denotes
the circle containing $A,B,C$. Therefore
\begin{eqnarray}\label{Chelsea-2}
\aangle\,(d_E,B'C')=\aangle\,{E_0}.
\end{eqnarray}
\noindent We have the pairs of lines $(d_D,d'_D)$, $(d_E,d'_E)$, $(d_F,d'_F)$ 
having the same slope as the line $BC$, and by symmetry, we see that 
$M'_4M'_2\parallel {M'_6M'_1}\parallel {M'_5M'_3}\parallel {B'C'}$. 
Hence, $\aangle\,(M'_4M'_2B')=\aangle\, (M'_2B'C')$ and 
$\aangle\,(B'M'_3C')=\pi-\aangle\,(M'_3B'C')-\aangle\,(M'_3C'B')
=\pi -\aangle\,{E_0}-\aangle\,{F_0}=\aangle\,{D_0}$. 
As a consequence, the quadrilateral $M'_4M'_2M'_1M'_3$ is inscribable. 
Since $\aangle\,(B'M'_5C')=\aangle\,(B'M'_3C')=\aangle\,{D_0}$, 
it follows that the quadrilateral $M'_4M'_2M'_1M'_6$ is also inscribable. 
The quadrilateral $M'_4M'_2M'_1M'_6$ is an isosceles trapezoid, hence it 
is inscribable as well. As a result, $M'_6$ lies on the circumscribed circle 
to the quadrilateral $M'_3M'_1M'_2M'_4$. Since $M'_5$ lies on the 
circumscribed circle to the quadrilateral $M'_6M'_1M'_3M'_5$, it follows 
that all the points $M_i$, $i=1,2,...,6$, lie on the same circle. 
\end{proof}

\noindent{\bf Remark~4.} Since the interior points lie on a circle, 
it follows that the exterior points lie on a circle (the transformed 
under the inversion of center $O$ and modulus $R$ of the former circle).

\vskip 0.08in
\noindent --------------------------------------
\vskip 0.10in

\noindent California Institute of Technology

\noindent MSC 700, Pasadena, CA 91126, USA

\noindent {\tt e-mail}: {\it amitrea\@@caltech.edu}

\end{document}